\documentclass[11pt,a4paper]{amsart}

\usepackage{amsthm,wrapfig,amsmath,amssymb,amsfonts,setspace,verbatim,graphicx,mathtools,mathrsfs,commath,float,mdframed,frame,xcolor}
\usepackage{enumitem}
\usepackage[hidelinks]{hyperref}
\usepackage{ulem}
\usepackage[margin=1.2in]{geometry}
\usepackage[T1]{fontenc}
\usepackage{url}

\makeatletter
\g@addto@macro{\UrlBreaks}{\UrlOrds}
\makeatother

\PassOptionsToPackage{hyphens}{url}\usepackage{hyperref}

\DeclareMathOperator{\dom}{dom}
\DeclareMathOperator{\id}{id}

\numberwithin{ut}{section}
\newtheorem{uq}{Question}
\newtheorem{ucl}{Claim}

\theoremstyle{remark}

\theoremstyle{definition}
\newtheorem{ue}{Example}

\begin{document}

\title{Dispersion points and rational curves}

\author{David S. Lipham}
\address{Department of Mathematics and Statistics, Auburn University, Auburn,
AL 36849, United States of America}
\email{dsl0003@auburn.edu}

\begin{abstract}
We construct two  connected plane sets which can be embedded into rational curves. The first is  a biconnected set with a dispersion point. It answers a question of Joachim Grispolakis. The second is indecomposable.   Both examples are completely metrizable.\end{abstract}

\maketitle

\section{Introduction}

\textit{Sur les ensembles connexes} \cite{kk}, published by B. Knaster and C. Kuratowski in 1921, was the first paper devoted to the study of connected sets in topology \cite{wil}.  The most surprising feature of the article was its presentation of an \textit{ensemble biconnexe}, known today as the ``Knaster-Kuratowski fan'' or ``Cantor's leaky tent''. The essence of the example is conveyed in  \cite{sa}: ``Cantor's leaky tent is connected, but just barely: remove one point and the whole thing falls apart.''  
 
The Knaster-Kuratowski fan has a certain countable structure. It is \textit{rational}, meaning it has a basis of open sets with countable boundaries. However, it  does not embed into any rational continuum \cite[Example 1]{tym}.\footnote{Tymchatyn's argument in \cite{tym} shows that in order for a connected subset $X$ of the Cantor fan to embed into a rational continuum,  the complement of $X$ must contain a dense union of arcs.}   Hints to  \cite[Exercise 1.4.C(c)]{eng1} reveal, moreover,  that  each of its  compactifications has a dense subspace homeomorphic to $\mathbb P\times (0,1)$, where $\mathbb P$ is the set of irrational numbers. In this sense, the Cantor fan's uncountable structure is  encoded in  the Knaster-Kuratowski fan. 

In the Houston Problem Book  \cite[Problem 79]{cookproblem}, Joachim Grispolakis asked: \textit{Is it true that no biconnected set with a dispersion point can be embedded into a rational continuum?}  The primary goal of this paper  is to show the answer is ``no''. We will construct a  biconnected subset of the Cantor fan which has a dispersion point and can be embedded into a rational continuum. 

By a technique in \cite{lip}, our counterexample  generates an  indecomposable connected  plane set  which also can be embedded into a rational continuum.   Indecomposable continua, by contrast, are known to  locally resemble  the Cantor set times the  unit interval.  And  every separator of an indecomposable continuum contains a Cantor set.  
 
For rational compactifications of our examples, we will rely on the following theorem by  E.D. Tymchatyn  and S.D. Iliadis: \textit{A separable metrizable space $X$ has a rational metrizable compactification if and only if $X$ has a basis of open sets with scattered boundaries}. See  \cite[Theorem 1]{tym} and  \cite[Theorem 8]{ili}.\footnote{The introduction to  \cite{ili} indicates there was a gap in the original proof of \cite[Theorem 1]{tym}.   \cite[Theorem 8]{ili}  is a stronger result with a corrected  proof.  A shorter proof of \cite[Theorem 8]{ili} is given in \cite{note}. } Our $X$'s will have bases of open sets with discrete boundaries.  Suslinian compactifications, which are \textit{almost} rational,  will be described in detail near the end of the paper.  Each will be the union of a zero-dimensional set and a   null-sequence of Hawaiian earrings.

\section{Preliminaries}

\subsection{Terminology}

A connected topological space $X$ is \textit{biconnected} provided $X$ cannot be written as the union of two disjoint non-degenerate connected subsets. 

A point $x$ in a non-degenerate connected space $X$ is a \textit{dispersion point} if $X\setminus \{x\}$ is hereditarily disconnected, i.e. if every connected subset of $X\setminus \{x\}$ is degenerate. Every connected set with a dispersion point is biconnected, but the converse is false under the Continuum Hypothesis \cite{mil}. 

A space X is \textit{rational} (\textit{rim-discrete}) provided $X$ has a basis  of open sets whose boundaries are  countable (discrete). If $X$ has a basis of clopen sets, then $X$ is  \textit{zero-dimensional}. 

A \textit{continuum} is a  connected compact metrizable space with more than one point.  

A continuum is \textit{Suslinian} if there is no uncountable collection of pairwise disjoint subcontinua \cite{lel}. Suslinian continua are frequently called \textit{curves} because they are $1$-dimensional. 

Rational continua are Suslinian. In fact, by second countability, each rational continuum can be written as the union of a zero-dimensional subspace and a countable set.  Every subcontinuum  must intersect the countable set.

A connected space $X$ is \textit{indecomposable} if $X$ cannot be written as the union of two proper closed connected subsets.  This is equivalent to saying every proper closed connected subset of $X$ is nowhere dense.  A connected space $X$ is \textit{widely-connected} if every proper closed connected subset of $X$ is degenerate. 
 
Biconnected and widely-connected spaces are known to be \textit{punctiform}, meaning they contain no continua.

\subsection{The Cantor set and Cantor fan} 

Let $C\subseteq [0,1]$ be the middle-thirds Cantor set.   For each $\sigma\in 2^{<\omega}$ define 
\begin{align*}0(\sigma)&=\textstyle\sum_{k=0}^{n-1} \frac{2\sigma(k)}{3^{k+1}}\text{; and} \\ 
1(\sigma)&=0(\sigma)+\textstyle\frac{1}{3^{n}},\end{align*}
where $n=\dom(\sigma)$.  The points $0(\sigma)$ and $1(\sigma)$ are called the ``endpoints'' of $C$.  More precisely, they are the endpoints of the  connected components of $\mathbb R \setminus C$. 

Let $B(\sigma)=[0(\sigma),1(\sigma)]\cap C$. The set of all $B(\sigma)$'s  is the canonical clopen basis for $C$.

Let $ \nabla:C\times[0,1]\to [0,1]^2$  be defined by  $\langle c,y\textstyle\rangle\mapsto\langle\frac{y(2c-1)+1}{2},y\rangle,$ so that $\nabla\restriction C\times(0,1]$ is a homeomorphism and $\nabla^{-1}\{\langle \frac{1}{2},0\rangle\}=C\times \{0\}$.  For any set $X\subseteq C\times (0,1]$  we put  $$\textbf{\d{$\nabla$}}X=\textstyle \nabla[X]\cup\textstyle \{\langle \frac{1}{2},0\rangle\}.$$
The \textit{Cantor fan}  is the set $\nabla(C\times [0,1])=\textbf{\d{$\nabla$}}(C\times (0,1])$.

\begin{figure}[h]
\centering
	\includegraphics[height=40mm,width=116mm]{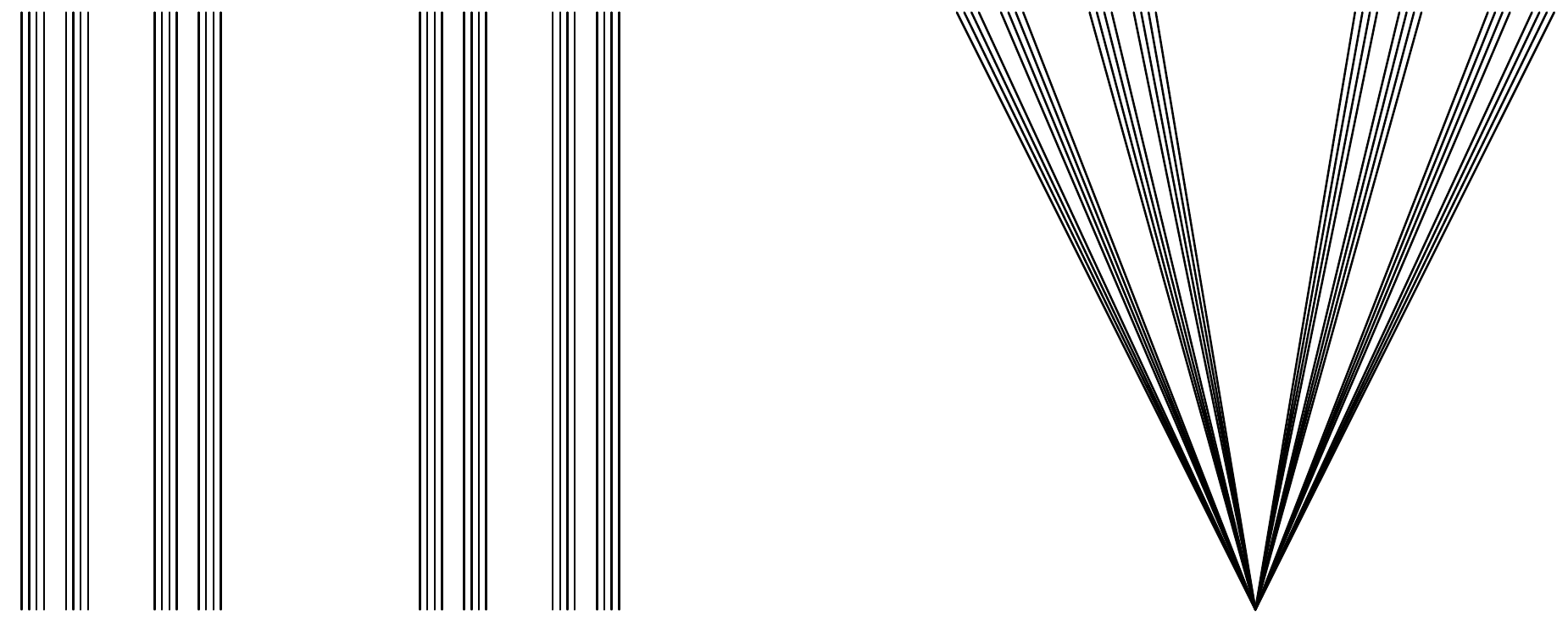}
	\caption{$C\times [0,1]$ and the Cantor fan}
\end{figure}

\subsection{Null-sequences and Hawaiian earrings} 

Let $\varrho$ denote the Euclidean metric on $\mathbb R^2$. A sequence of sets $A_0,A_1,...\subseteq \mathbb R ^2$ is a \textit{null-sequence} provided the diameter (with respect to $\varrho$) of $A_n$ converges to $0$ as $n\to\infty$. A continuum $H$ is called a \textit{Hawaiian earring} if there is a null-sequence of simple closed curves $J_0,J_1,...\subseteq \mathbb R ^2$ such that $H=\bigcup \{J_n:n<\omega\}$ and $|\bigcap \{J_n:n<\omega\}|=1$.

\section{Main examples}

\begin{ue}Our goal  is to construct a  hereditarily disconnected $X\subseteq C\times (0,1)$ such that $\textbf{\d{$\nabla$}}X$ is rim-discrete and connected. To accomplish this,   we will reduce  $C\times (0,1)$ to a zero-dimensional set $P$ by deleting  $\omega$-many pairwise disjoint copies of a  compact set $D\subseteq C\times [0,1]$.  Then  $X$ will be obtained by adding back a countable point  set $Q$.

Throughout this example, $\pi_0$ and $\pi_1$ are the coordinate projections in $C\times\mathbb R$ defined by $\pi_0(\langle c,r\rangle)=c$ and $\pi_1(\langle c,r\rangle)=r$.

We begin with a function which was introduced by Wojciech D\k{e}bski in \cite{deb}. Let $\{d_n:n<\omega\}\subseteq C$ be a  dense set of non-endpoints. Each $d_n$ is the limit of both increasing and decreasing sequences of points in $C$. Define   $f:C\to [0,1]$ by 
$$f(c)\hspace{2mm}=\hspace{-3mm}\sum_{\{n<\omega:d_n<c\}} \hspace{-5mm}2^{-n-1}.$$ 

Observe that $f$ is  non-decreasing, and $\{d_n:n<\omega\}$ is the set of (jump) discontinuities of $f$; 
\begin{align*}
&\lim\limits_{c\to d_n^{\;-}}f(c)=r_n:=f(d_n)\text{, whereas}\\
&\lim\limits_{c\to d_n^{\;+}}f(c)=s_n:=f(d_n)+2^{-n-1}.
\end{align*}
 
Let  $\textstyle D=\{\langle c,f(c)\rangle:c\in C\}\cup \bigcup \{\{d_n\}\times [r_n,s_n]:n<\omega\}$ be the graph of $f$ together with the vertical arcs from $\langle d_n,r_n\rangle$ to $\langle d_n,s_n\rangle$. Then: 
\begin{itemize}\renewcommand{\labelitemi}{\tiny$\square$}
	\item $D$ is compact; 
	\item $\pi_1[D]=[0,1]$;
	\item  $C\times [0,1]\setminus D$ is the union of two disjoint open sets $\{\langle c,r\rangle\in C\times[0,1]:r<f(c)\}$ and $$\;\;\;\;\;\;\;\;\;\{\langle c,r\rangle\in C\times[0,1]:(r>f(c)\text{ if }c\notin\{d_n:n<\omega\}) \text{ or } (r>s_n\text{ if }c=d_n)\}$$
consisting of all half-open intervals ``below'' $D$ and ``above'' $D$, respectively. In fact,   $D$  extends to a Jordan arc  
which separates $[0,1]^2$ in this manner. 
\end{itemize}

Say that a collection of sets $\mathcal R$ is a \textit{partial tiling} of $C\times \mathbb R$ if: for every $R\in \mathcal R$ there exists $\sigma\in 2^{<\omega}$ and two numbers $a<b\in \mathbb R$ such that  $R=B(\sigma)\times [a,b]$; and  for every two elements $R_1\neq R_2$ of $\mathcal R$ there exists $\sigma\in 2^{<\omega}$ and $c\in \mathbb R$ such that  $R_1\cap R_2\subseteq B(\sigma)\times \{c\}$.

We will recursively define  a sequence $\mathcal R_0,\mathcal R_1, ...$  of finite partial tilings of $C\times \mathbb R$ so that for each $n<\omega$:

\begin{itemize} 
	\item[(i)] $\mathcal R_n$ consists of rectangles $R^n_i=B(\sigma^n_i)\times [a^n_i,b^n_i]$, where $\sigma^n_i\in 2^n$ and $i<|\mathcal R_n|<\omega$;
	\item[(ii)]$0<b^n_i-a^n_i\leq\frac{1}{n+1}$ for all $i<|\mathcal R_n|$;
	\item[(iii)]the linear transformations  $T^n_i:C\times [0,1]\to R^n_i$ defined by  $$T^n_i(\langle c,r\rangle)=\textstyle\langle 0(\sigma^n_i)+\frac{c}{3^n}, a^n_i+r(b^n_i-a^n_i)\rangle$$ are such that $T^n_i[D]\cap T^k_j[D]=\varnothing$ whenever $k< n$ or $i\neq j$; 
	\item[(iv)]  for each $\sigma\in 2^n$, 
$$\pi_1\big[\textstyle B(\sigma)\times \mathbb R\cap \bigcup \{T^k_i[D]:k\leq n\text{ and }i<|\mathcal R_k|\}\big]=[-n,n+1];$$
	\item[(v)] for every arc  $I\subseteq C\times[-n,n+1]\setminus \bigcup \{T^k_i[D]:k\leq n\text{ and }i<|\mathcal R_k|\}$    there are  integers $i<|\mathcal R_n|$,  $k\leq n$, and $j<|\mathcal R_{k}|$ such that $I\subseteq R^n_{i}\cup R^{k}_{j}$ and $$\textstyle d(I,T^{k}_{j}[D]):=\min\{\varrho(x,y):x\in I\text{ and }y\in T^k_j[D]\}<\frac{1}{n+1}+\frac{1}{3^n}.$$  Here,  $\varrho$ is the Euclidean metric on $\mathbb R ^2$.
\end{itemize}

\noindent Note that  $T^n_i$ is just the natural homeomorphism between $C\times[0,1]$ and $R^n_i$.

\noindent\begin{itemize}
\item[$n=0$:] Let $\mathcal R_0=\{R^0_0\}=\{C\times [0,1]\}=\{B(\varnothing)\times [0,1]\}$, and $T^0_0=\id_{C\times[0,1]}$.

\item[$n=1$:]  Tile  $([0,\frac{1}{3}]\cap C)\times [f(\frac{1}{3}),1]$ and $([\frac{2}{3},1]\cap C)\times [0,f(\frac{2}{3})]$ with   four rectangles:
\begin{align*}
	R^1_0&=B(\langle 0\rangle)\times[\textstyle\frac{1}{2}(f(\frac{1}{3})+1),1];\\
	R^1_1&=B(\langle 0\rangle)\times [\textstyle f(\frac{1}{3}),\textstyle\frac{1}{2}(f(\frac{1}{3})+1)];\\
	R^1_2&= B(\langle 1\rangle)\times [\textstyle\frac{1}{2}\textstyle f(\frac{2}{3}), f(\frac{2}{3})];\\
	R^1_3&=  B(\langle 1\rangle)\times [\textstyle 0,\textstyle\frac{1}{2}f(\frac{2}{3})].
\end{align*}
Let $R^1_4,R^1_5,...,R^1_{11}$ enumerate the set of eight rectangles $$\textstyle\{B(\sigma)\times [a,a+\frac{1}{2}]:\sigma=\langle 0\rangle,\langle 1\rangle\text{ and }a=-1,-\frac{1}{2},1,\frac{3}{2}\}.$$ Let $\mathcal R_1=\{R^1_i:i<12\}$.
\begin{figure}[H]
	\includegraphics[scale=0.5]{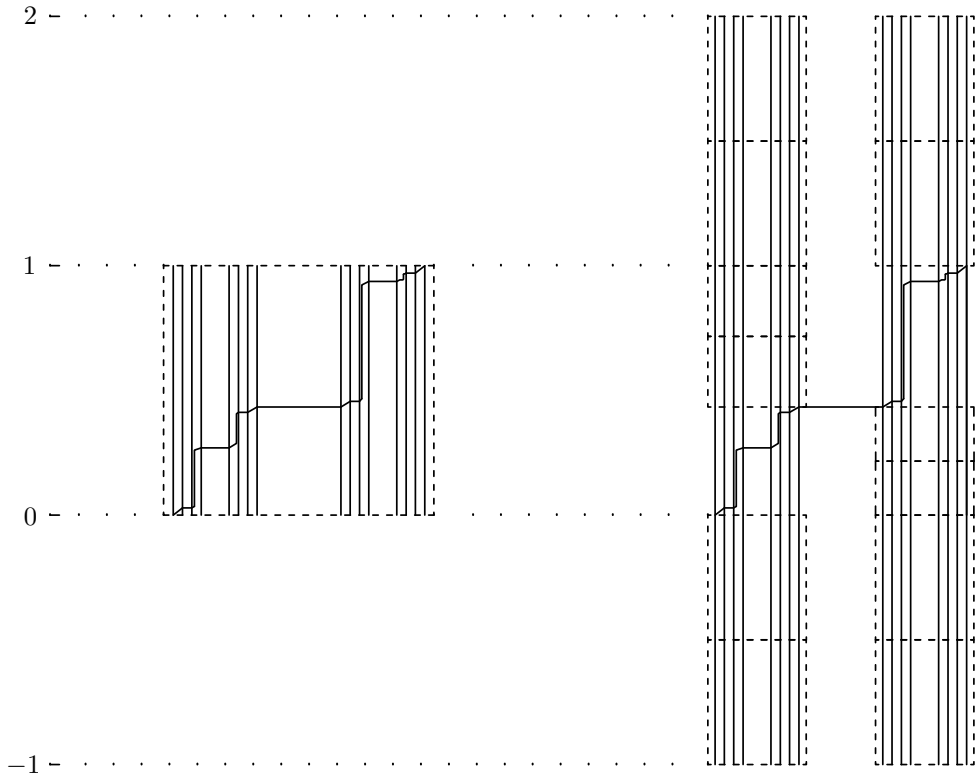}
	\caption{Jordan arc extending $D$ in  $R^0_0$ (left),  and the rectangles of $\mathcal R_1$ (right).}
	\label{366}
\end{figure}
\item[$n\geq 2$:] Assume  $\mathcal R_k$ has  been  defined for all $k<n$. 

Fix $\sigma\in 2^n$, and partially tile  $B(\sigma)\times [-n,n+1]$ in the following way. The sets 
\begin{align*}
	&\textstyle\pi_1[\{0(\sigma)\}\times \mathbb R \cap \bigcup \{T^k_i[D]:k< n\text{ and }i<|\mathcal R_k|\}];\text{ and}\\
	&\textstyle\pi_1[\{1(\sigma)\}\times \mathbb R\cap \bigcup \{T^k_i[D]:k<n\text{ and }i<|\mathcal R_k|\}] 
\end{align*}
can be enumerated $x_0,...,x_m$ and $y_0,...,y_m$, respectively,  such that $$-n+1\leq x_0<y_0<x_1<y_1<...<x_m<y_m\leq n$$  and $x_j$ and $y_j$ belong to the same $T^k_i[D]$. Note that $B(\sigma)\times [x_j,y_j]$ is contained in the closed $\frac{1}{3^n}$-neighborhood of the unique $T^k_i[D]$ it intersects. Tile each strip 
\begin{align*}
	&B(\sigma)\times [-n,x_0];\\
	&B(\sigma)\times [y_m,n+1]\text{; and} \\
	&B(\sigma)\times [y_j,x_{j+1}]\text{, }j<m, 
\end{align*}
with finitely many rectangles of the form $B(\sigma)\times [a,b]$  where  $0<b-a\leq\frac{1}{n+1}$.  
 
Repeat for each $\sigma\in 2^n$ to obtain  $\mathcal R_n$.
 
Note that conditions (i) through (iv) are satisfied.  We have also guaranteed that for all $l,m<n$, $j<|\mathcal R_l|$,  $k<|\mathcal R_m|$, and $\sigma\in 2^n$,   there exists $i<|\mathcal R_n|$ such that $T^n_i[D]$ is between $T^l_j[D]\cap B(\sigma)\times \mathbb R$ and $T^m_k[D]\cap B(\sigma)\times \mathbb R$.  Together with (iv), this implies (v).  
\end{itemize}

Now let $M=\{\langle d_n,\textstyle\frac{r_n+s_n}{2}\rangle:n<\omega\}=\{\langle d_n, f(d_n)+2^{-n-2}\rangle:n<\omega\}
$ be the set of midpoints of  vertical arcs in $D$.  Note that $M$ is discrete.  Define $$Y=C\times \mathbb R\setminus \bigcup \{ T^n_i[D\setminus M]:n<\omega \text{ and }i<|\mathcal R_n|\}.$$ In other words, $Y=P\cup Q$ where $$P=C\times \mathbb R\setminus \bigcup \{T^n_i[D]:n<\omega\text{ and }i<|\mathcal R_n|\}$$ and $$Q=\bigcup \{ T^n_i[M]:n<\omega\text{ and }i<|\mathcal R_n|\}.$$

Let $\Xi:C\times \mathbb R\to C\times (0,1)$ be the  homeomorphism  defined by $$\textstyle\Xi(\langle c,r\rangle)=\langle c,\frac{1}{\pi}\arctan(r)+\frac{1}{2}\rangle.$$  Finally,  $X=\Xi[Y].$

\begin{ucl}$X$ is hereditarily disconnected.\end{ucl}
 
\begin{proof}[Proof of Claim 1]Suppose not.  Then there is a vertical arc  $I\subseteq Y$ of positive height. Let $N<\omega$ be such that $I\subseteq C\times [-N,N+1]$.  
  
Apparently, $I\cap Q=\varnothing$ because each point $\langle c,r\rangle\in Q$ is isolated in its strand $Y\cap \{c\}\times\mathbb R$. Thus $I\subseteq P$.   For each $n\geq N$ let $i(n)<|\mathcal R_n|$,  $k(n)\leq n$, and $j(n)<|\mathcal R_{k(n)}|$ be given by condition (v), so that $$I\subseteq R^n_{i(n)}\cup R^{k(n)}_{j(n)}$$ and $d(I,T^{k(n)}_{j(n)}[D])<\frac{1}{n+1}+\frac{1}{3^n}$. See Figure 3.

The total height of $R^n_{i(n)}\cup R^{k(n)}_{j(n)}$ is less than or equal to $\frac{2}{k(n)+1}$, so $\{k(n):n<\omega\}$ must be  finite.   On the other hand, $$\textstyle 0<d(I,T^{k(n)}_{j(n)}[D])<\frac{1}{n+1}+\frac{1}{3^n}$$ implies $\{k(n):n<\omega\}$ is infinite.  We have a contradiction.\end{proof}

\begin{ucl}\bf{\d{$\nabla$}}$X$ \textnormal{\textit{is connected.  }}\end{ucl}

\begin{proof}[Proof of Claim 2]Note that $\pi_0[Q]$ is countable, as is $\{c\}\times \mathbb R\setminus Y$ for every $c\in C\setminus \pi_0[Q]$. In particular, $\overline{Y\cap\{c\}\times \mathbb R}=\{c\}\times\mathbb R$ for each $c\in C\setminus \pi_0[Q]$, and $$\overline{Y\cap (C\setminus \pi_0[Q])\times\mathbb R}=C\times\mathbb R.$$

Supposing $\textbf{\d{$\nabla$}}X$ is not connected, there exists a clopen partition $\{A,B\}$ of $Y$ and a point $c\in C\setminus \pi_0[Q]$  such that $A\cap \{c\}\times \mathbb R$ and $B\cap \{c\}\times \mathbb R$ are non-empty. Since $\{c\}\times \mathbb R\subseteq\overline A\cup \overline B$, we have  $\overline A\cap \overline B\cap \{c\}\times \mathbb R\neq\varnothing$. Applying the Baire Category Theorem in  $(C\setminus \pi_0[Q])\times \mathbb R\cap \overline A\cap \overline B$, we find there exists $n,i<\omega$ and a $C\times \mathbb R$-open set $U\times (r,s)$ such that $$\varnothing \neq  (U\setminus \pi_0[Q])\times (r,s)\cap \overline A\cap \overline B\subseteq T^n_i[D].$$ Let $\langle c',r'\rangle$ be an element of this intersection. 
   
The interval $\{c'\}\times (r',s)$  is contained in either $\overline A\setminus \overline B$ or $\overline B\setminus \overline A$, and likewise for  $\{c'\}\times (r,r')$.   But  $\{c'\}\times [(r,r')\cup (r',s)]$ is contained in neither  $\overline A$ nor $\overline B$. Without loss of generality, $\{c'\}\times (r',s)\subseteq \overline A\setminus \overline B$ and  $\{c'\}\times (r,r')\subseteq \overline B\setminus \overline A$.  
   
In $(U\setminus \pi_0[Q])\times (r,s)$, all intervals above $T^n_i[D]$ with $C$-coordinates sufficiently close to $c'$ belong to $\overline A\setminus \overline B$, and all lower segments eventually belong to $\overline B\setminus \overline A$. But there exists $m<\omega$ such that $\pi_0(T^n_i(\langle d_m,0\rangle))$ is very close to $c'$ and $T^n_i[\{d_m\}\times [r_m,s_m]]\subseteq U\times (r,s)$.  See Figure 4. We have $T^n_i(\langle d_m,\frac{r_m+s_m}{2}\rangle)\in A\cap B$, a contradiction. \end{proof}
  
 \begin{figure}[h]
	\centering
	\includegraphics[scale=0.5]{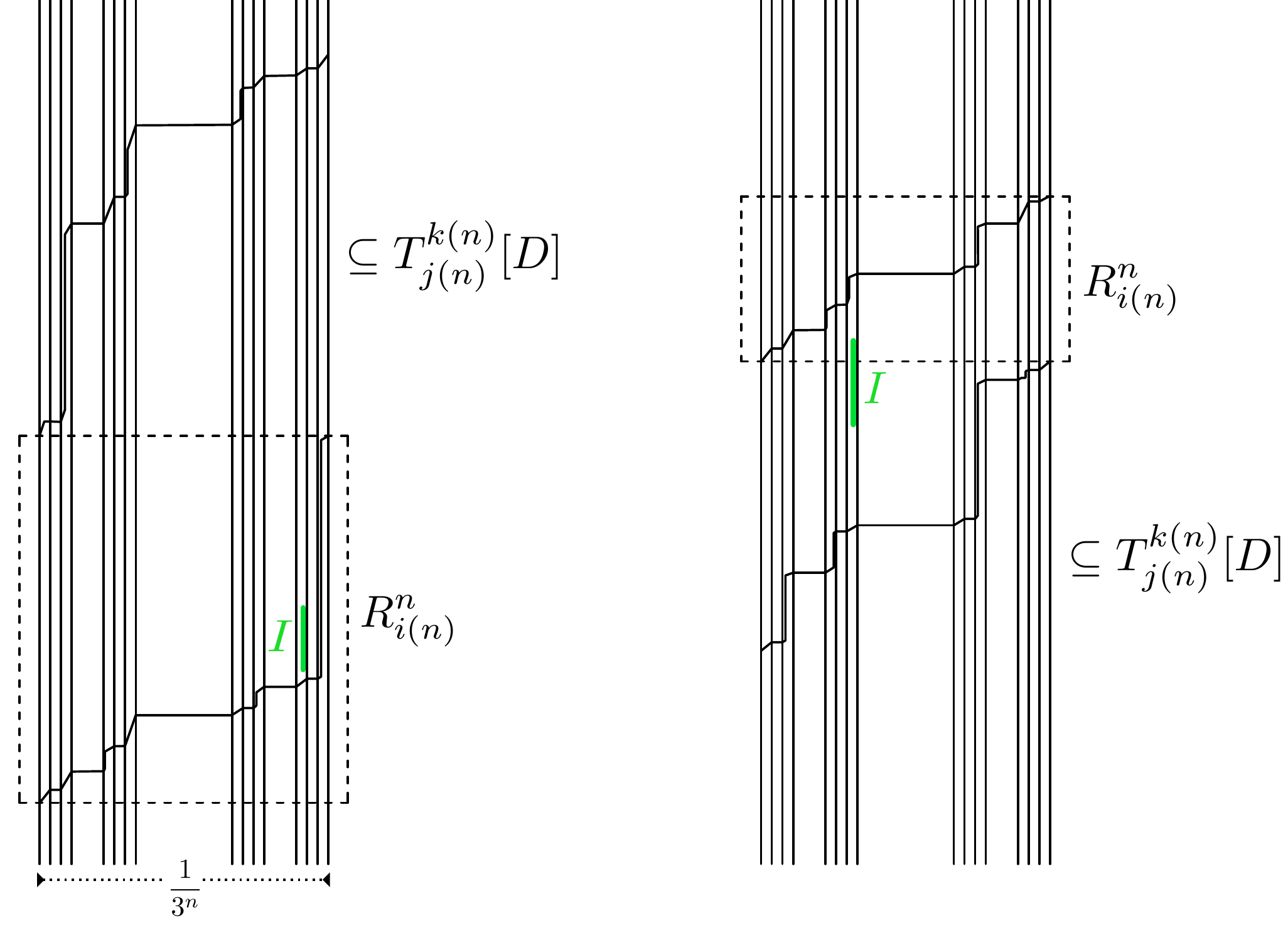} 
	\caption{Possible locations of $I$.}  
	\label{fig:prob1_6_1}
\end{figure}

\begin{figure}[h]
	\centering
	\hspace{1cm}\includegraphics[scale=0.3]{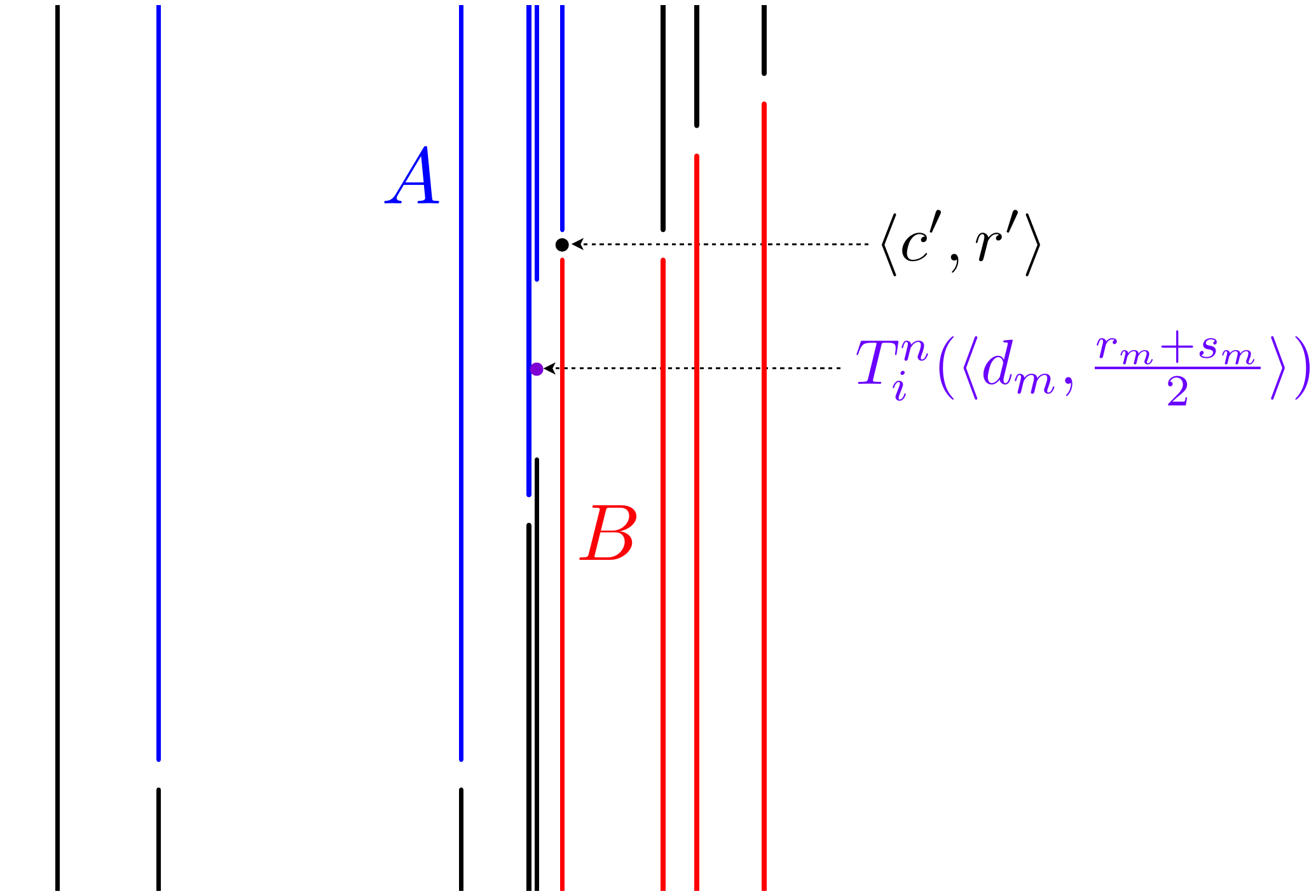}
	\caption{Proving $\textbf{\d{$\nabla$}}X$ is connected}
\end{figure}

\begin{ucl}$Y$ is rim-discrete. \end{ucl}

\begin{proof}[Proof of Claim 3]  Let $U\times V$ be any $C\times \mathbb R$-neighborhood of a point $\langle c,r\rangle\in Y$. 

\underline{Case 1:}   $c\notin \pi_0[Q]$. By Claim 1,  $\langle c,r\rangle$ is the limit of increasing and decreasing sequences of points in $\{c\}\times\mathbb R$ from infinitely many $T^n_i[D]$'s.  So there exist $n,m,i,j<\omega$ such that $R^n_i\cup R^m_j\subseteq U\times V$, $c\in B(\sigma^n_i)\cap B(\sigma^m_j)$, and $b^n_i<r<a^m_j$. Let $W$ be the set of points in $X$ between $T^n_i[D]$ and $T^m_j[D]$. Then $W$ is an open subset of $X$, and $\langle c,r\rangle\in W\subseteq U\times V$.  Finally, $\partial _Y W$ is discrete because it is contained in the discrete set $T^n_i[M]\cup T^m_j[M]$.

\underline{Case 2:} $c\in \pi_0[Q]$.  There   is an interval $[a,b]\subseteq V$ such that  $a<r<b$ and $$X\cap \{c\}\times [a,b]=\{\langle c,r\rangle\}.$$  For in this case there are integers $n,i,m<\omega$ such that $\langle c,r\rangle =T^n_i(\langle d_m,\frac{r_m+s_m}{2}\rangle)$. Assume $V$ is a closed interval, and let  $$[a,b]=V\cap [\pi_1(T^n_i(\langle d_m,r_m\rangle )),\pi_1(T^n_i(\langle d_m,s_m\rangle))].$$  

Let $N$ be an integer such that $b^n_i-a^n_i<\frac{r-a}{2}$ whenever $n\geq N$.   Only finitely many vertical arcs in $D$ have height at least $\frac{r-a}{2}$, and there are only finitely many $T^k_i[D]$'s with $k<N$. It follows that $F_0:=\{e\in C:\{e\}\times (a,\textstyle\frac{r+a}{2})\cap X\subseteq Q\}$ is finite. For if $e\in F_0$ then  $\{e\}\times (a,\frac{r+a}{2})$ is covered by a countable collection of pairwise disjoint arcs and singletons. So $\{e\}\times (a,\frac{r+a}{2})$ must be contained in one of the arcs; a single vertical arc of some $T^k_i[D]$ with $k<N$.  

Likewise,   $F_1:=\{e\in C:\{e\}\times (\textstyle\frac{r+b}{2},b)\cap X\subseteq Q\}$ is finite.

Let $A_0\supseteq A_1\supseteq  ...$ be a decreasing sequence of $C$-clopen sets such that $A_0\subseteq U\setminus [(F_0\cup F_1)\setminus \{c\}]$ and  $\{c\}=\bigcap \{A_k:k<\omega\}$.  Let $$\textstyle\Pi^0:=\{\langle n,i\rangle:B(\sigma^n_i)\times [a^{n}_i,b^{n}_i]\subseteq U\times (a,\frac{a+r}{2})\}.$$ By the case $c\notin \pi_0[Q]$ we have $A_0\setminus \{c\}\subseteq \bigcup \{B(\sigma^n_i):\langle n,i\rangle\in \Pi^0\}$.  Moreover, for each $k<\omega$ there is a finite $\Pi_k^0\subseteq\Pi^0$ such that $A_{k}\setminus A_{k+1}\subseteq \bigcup \{B(\sigma^n_i):\langle n,i\rangle\in \Pi^0_k\}$. Let 
\begin{align*}
W_0=\Big\{\langle e,s\rangle \in X:\big(\exists k<\omega\text{ and } \langle n,i\rangle\in \Pi^0_k\big)\big(e\in &B(\sigma^n_i)\cap A_{k}\setminus A_{k+1}\text{ and } \\
&s>\max\pi_1\big[\{e\}\times \mathbb R\cap T^{n}_{i}[D]\big]\big)\Big\}.       
\end{align*}

We may define $\Pi^1$ and finite $\Pi^1_k$'s similarly for  $U\times (\frac{r+b}{2},b)$, and then let 
	\begin{align*}
	W_1=\Big\{\langle e,s\rangle \in X:\big(\exists k<\omega\text{ and }\langle n,i\rangle\in \Pi^1_k\big)\big(e\in &B(\sigma^n_i)\cap A_{k}\setminus A_{k+1}\text{ and }\\
	&s<\min\pi_1\big[\{e\}\times \mathbb R\cap T^{n}_{i}[D]\big]\big)\Big\}.
	\end{align*}

Observe that $W:=(W_0\cap W_1)\cup \{\langle c,r\rangle\}$ is an $X$-open subset of $U\times V$ and $\partial _Y W$  is discrete.\end{proof}

\begin{ucl}\bf{\d{$\nabla$}}$X$ \textnormal{\textit{is rim-discrete.}} \end{ucl}

\begin{proof}[Proof of Claim 4] By Claim 3 we only need to show \textbf{\d{$\nabla$}}$X$ has a basis of open sets with discrete boundaries at the point  $\langle \frac{1}{2},0\rangle$. To that end, let $\varepsilon\in (0,1)$. We will show there is a $\textbf{\d{$\nabla$}}X$-open set $\textbf{\d{$\nabla$}}U$ such that $\langle \frac{1}{2},0\rangle\in \textbf{\d{$\nabla$}}U\subseteq \textbf{\d{$\nabla$}}[C\times (0,\varepsilon)]$ and $\partial _{\textbf{\d{$\nabla$}}X} \textbf{\d{$\nabla$}}U$ is discrete. 

Let $r=\tan(\pi(\varepsilon-\frac{1}{2}))$. The arguments in Claim 3 show there is a finite collection of $B(\sigma^n_i)$'s which cover $C$ and satisfy   $B(\sigma^n_i)\times [a^n_i,b^n_i]\subseteq C\times (-\infty, r)$.   Let $W$ be the set of all points in $Y$ below the corresponding $T^n_i[D]$'s, and put $U=\Xi[W]$.  Then  $\textbf{\d{$\nabla$}}U$ is as desired.  \end{proof}

Note that $D\setminus M$ is $\sigma$-compact, so $Y$ is a $G_\delta$-subset of $C\times \mathbb R$. Hence both $X$ and $\textbf{\d{$\nabla$}} X$ are completely metrizable. 

This concludes Example 1.\end{ue}

\begin{ue}There is a rim-discrete widely-connected $G_\delta$-subset of the plane.

Let $X$ be as defined in Example 1.  In the bucket-handle continuum $K$ there is a Cantor set $\Delta$ such that $K\setminus \Delta$  is the union of $\omega$-many pairwise disjoint open sets $K_i\simeq C\times (0,1)$.   For each $i<\omega$,  let $\varphi_i:C\times(0,1)\to K_i$ be a homeomorphism, and put $X_\omega=\bigcup \{\varphi_i[X]:i<\omega\}$. Essentially,  $X_\omega\cup \Delta$ is the set $W[X]$ from \cite{lip}.  By the particular construction in \cite{lip}, we may assume $(K_i)$ is a null-sequence; this will be important later in Example 4.

Recall $X$ is a dense hereditarily disconnected subset of  $C\times(0,1)$, and $\textbf{\d{$\nabla$}} X$ is connected. By  \cite[Proposition 1]{lip}, $X_\omega\cup \Delta$ widely-connected.  By the closing remark in Example 1 and \cite[Proposition 6]{lip}, $X_\omega\cup \Delta$ is completely metrizable.  And by treating each point of $\Delta$  like  $\langle \frac{1}{2},0\rangle$ in  Claim 4,  we  see that $X_\omega\cup \Delta$  is rim-discrete. 
\end{ue}

\section{Suslinian compactifications}

Here we define Suslinian upper semi-continuous decompositions of the Cantor fan and  Knaster bucket-handle.  The respective continua will homeomorphically contain \textnormal{\textbf{\d{$\nabla$}}}$X$ and $X_\omega\cup \Delta$.

\begin{ue}There is a Suslinian continuum $\mathcal F\supseteq \textnormal{\textbf{\d{$\nabla$}}}X$.

By Claim 3 of Example 1,  there is a null-sequence $D_0,D_1,...\subseteq \nabla(C\times(0,1))$ of disjoint copies of $D$ such that $\nabla(C\times(0,1))\setminus \bigcup \{D_n:n<\omega\}$ is zero-dimensional, and $\nabla X\cap \bigcup \{D_n:n<\omega\}$ is  a countable point set (the set of all midpoints of arcs in the $D_n$'s). 

Let $f$ be the D\k{e}bski function from the beginning of Example 1.  Observe that  $$E:=\{\langle c,f(c)\rangle:c\in C\}\cup \{\langle d_n,s_n\rangle:n<\omega\}$$ is the $C\times[0,1]$-closure of the graph of $f$;  $E=\overline{\{\langle c,f(c)\rangle:c\in C\}}$. Thus, $E$ is compact. 

Let $F$ denote the the Cantor fan $\nabla(C\times[0,1])$.

For each $n<\omega$ let $E_n\subseteq D_n$ be the natural copy of $E$ in $D_n$, so that $E_n\cap \textnormal{\textbf{\d{$\nabla$}}}X=\varnothing$ for each $n<\omega$. Identify each $E_n$ with the singleton $e_n:=\{E_n\}$, and put $\mathcal F=(F\setminus \bigcup \{E_n:n<\omega\})\cup \{e_n:n<\omega\}$.  Since $(E_n)$ is a null-sequence of compact sets, $\mathcal F$ is an upper semi-continuous decomposition of $F$. Therefore $\mathcal F$ is a continuum.  Additionally, $\mathcal F$  contains \textnormal{\textbf{\d{$\nabla$}}}$X$. It remains to show $\mathcal F$ is Suslinian. 
 
For each $n<\omega$, let  $\mathcal D_n=(D_n\setminus E_n)\cup \{e_n\}\subseteq \mathcal F$ denote the quotient of $D_n$ obtained by shrinking $E_n$ to $e_n$.  We note that $\mathcal D_n$ is (homeomorphic to) a Hawaiian earring.

\begin{ucl}Each subcontinuum of $\mathcal F\setminus\{e_n:n<\omega\}$ is   an arc  in some $\mathcal D_m$. \end{ucl}

\begin{proof}[Proof of Claim 5]Let $\mathcal S\subseteq \mathcal F\setminus\{e_n:n<\omega\}$ be a continuum. 

Since $\mathcal F\setminus \bigcup \{\mathcal D_n:n<\omega\}$ is zero-dimensional, there exists $m<\omega$ such that $S\cap \mathcal D_m\neq \varnothing$.   This alternatively follows from the fact that $X$ is punctiform. 

Let $s\in \mathcal S\cap \mathcal D_m$.  Then there is a simple closed curve $\mathcal J\subseteq \mathcal D_m$ such that $s\in \mathcal J\setminus \{e_m\}$.  Note that  $\mathcal J$ corresponds to a vertical jump $J\subseteq D_m$.
 
Let $k<\omega$ be the unique integer such that $(\nabla\circ \Xi)^{-1}(D_m)$ spans a rectangle from  $\mathcal R_k$. For each $n<\omega$ let $i(n),j(n)< |\mathcal R_{k+1+n}|$ be such that:
	\begin{align*}
  		&R^{k+1+n}_{i(n)}\text{ is the lowest member of }\mathcal R_{k+1+n}\text{ above }J\text{, and}\\
		&R^{k+1+n}_{j(n)}\text{ is the highest member of }\mathcal R_{k+1+n}\text{ below }J. 
	\end{align*}
 
Let $U_n$ be the open region of $F$ between $D^0_n:=(\nabla\circ \Xi)(T^{k+1+n}_{i(n)}[D])$ and $D_m$.  
 
Let $V_n$ be the open region of $F$ between $D^1_n:=(\nabla\circ \Xi)(T^{k+1+n}_{j(n)}[D])$ and $D_m$.  
 
Let $a$ and $b$ be the top and bottom points of $J$. The set $$G_n:=U_n\cup V_n\cup J\setminus \{a,b\}$$ is open in $F$.  By construction $G_n$ is also a union of elements of $\mathcal F$, and $$\partial G_n\subseteq D^0_n\cup D^1_n\cup D_m.$$   
  
Let $\mathcal W$ be an open subset of $\mathcal F$ such that $e_m\in \mathcal W$ and $\mathcal W\cap \mathcal S=\varnothing$.   Then $E_m\subseteq W$, $D^0_n\to a\in W$, $D^1_n\to b\in W$, and  $\mathcal W$ contains all but finitely many simple closed curves in $\mathcal D_m$.   So there is an integer $N<\omega$ such that $\partial G_n \subseteq W$ for all $n\geq N$. 

Note that $H_n:=\overline{G_n} \cup E_m$ is a closed subset of $F$, and $H_n$ is a union of elements of $\mathcal F$. For any $n\geq N$ we have 
$$\partial_{\mathcal F} \mathcal G_n= \mathcal H_n\setminus  \mathcal G_n=\mathcal D^0_n\cup \mathcal D^1_n\cup  (\partial U_n\cap \partial V_n\setminus E_m)\cup \{e_m\}\subseteq \mathcal W,$$ so $\mathcal S$ misses the boundary of $\mathcal G_n$.  Since $\mathcal S$ is connected and $s\in \mathcal G_n$, it follows that    $\mathcal S\subseteq \mathcal G_n$. In conclusion,  $$\mathcal S\subseteq \bigcap\{\mathcal G_n:n<\omega\}=\mathcal J\setminus \{e_m\}.$$   Therefore $\mathcal S$ is an arc in $\mathcal D_m$. \end{proof}

By Claim 5, $\mathcal F$ is the union of a punctiform set $\mathcal P$ and a countable set $\mathcal Q$.  Indeed, let $\mathcal Q\supseteq  \{e_n:n<\omega\}$ be a countable set which contains a dense subset of each $\mathcal D_n$, and put $\mathcal P=\mathcal F\setminus \mathcal Q$. Every subcontinuum of $\mathcal F$ must intersect $\mathcal Q$, so $\mathcal F$ is Suslinian. 

This concludes Example 3.\end{ue}

\begin{figure}[h]
\centering
    \includegraphics[scale=0.8]{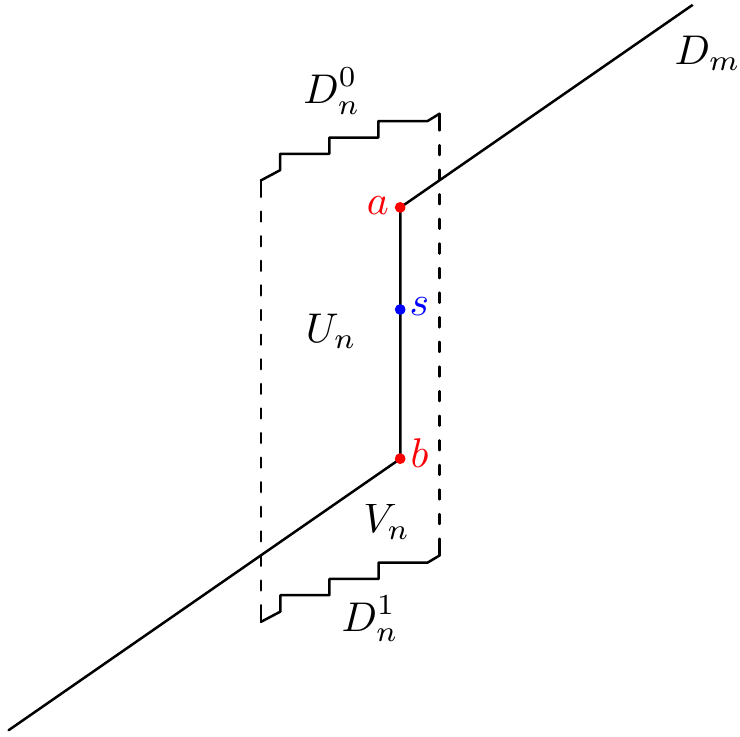}\;\;\;\;\;\;\;\;\;\;\;\;\;\;\;\;\;\;
    \includegraphics[scale=0.8]{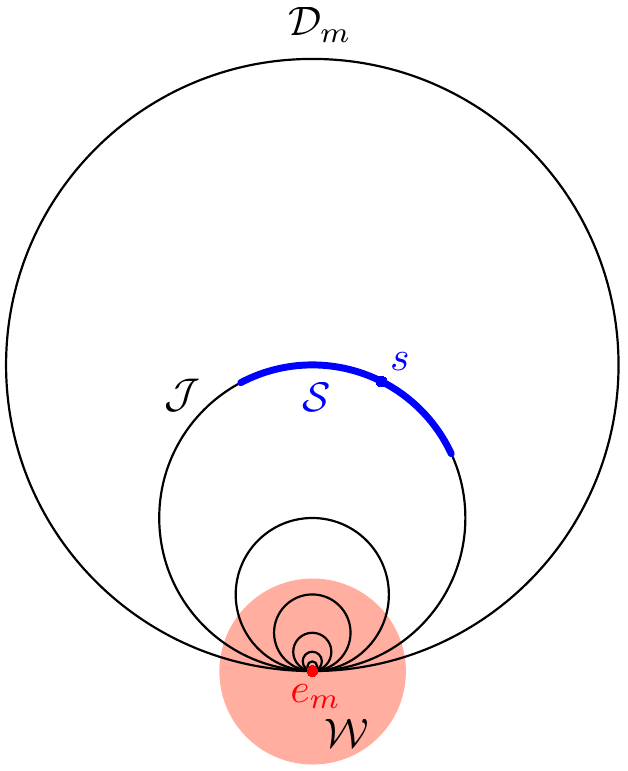}
    \caption{Illustration of Claim 5}
    \label{366}
\end{figure}

\begin{ue}There is  a Suslinian continuum $\mathcal K \supseteq X_\omega\cup \Delta$.

According to the construction in Example 2, $X_\omega$ is the union of a  null-sequence of disjoint open sets   homeomorphic to $X$.  Thus,  there is a null-sequence $D_0,D_1,...\subseteq K\setminus \Delta$ of disjoint copies of $D$  such that $K\setminus \bigcup \{D_n:n<\omega\}$ is zero-dimensional and $X_\omega \cap \bigcup \{D_n:n<\omega\}$ is  countable.   As in Example 3, it is now possible to define a Suslinian upper semi-continuous decomposition $\mathcal K$ of $K$, such that $X_\omega\cup \Delta\subseteq \mathcal K$.  
\end{ue}

We remark that each example in this section is  the union of a punctiform set and a countable set.  This is materially stronger than the Suslinian property by Piotr Minc's example in \cite{minc}.

\section{Conclusion}

The continua in Examples 3 and 4 are not rational.  Arguments in Claim 2 of Example 1 show that, in fact, every countable separator of $\mathcal F$ (Example 3) contains the vertex $\langle \frac{1}{2},0\rangle$.  Likewise,  every co-countable subset of $\mathcal K$ (Example 4) is connected. 
Rational compactifications of  \textnormal{\textbf{\d{$\nabla$}}}$X$ and $X_\omega\cup \Delta$  exist nevertheless  by  \cite[Theorem 1]{tym}, or the stronger  \cite[Theorem 8]{ili}. 
 
A point $e$ in a connected space $X$  is called an \textit{explosion point} if for every two points $x$ and $y$ in $X\setminus\{e\}$, there is a relatively clopen subset of $X\setminus \{e\}$ which contains $x$ but not $y$.  The dispersion point $\langle \frac{1}{2},0\rangle$ in Example 1 is not an explosion point because $\textnormal{\textbf{\d{$\nabla$}}}X\setminus \{\langle \frac{1}{2},0\rangle\}$ has non-degenerate quasi-components. Thus, the following variation of \cite[Problem 79]{cookproblem} is still open.

\begin{uq}Is it true that no biconnected set with an explosion point can be embedded into a rational continuum?\end{uq}

\subsection*{Added September 2019.}Question 1 was recently answered by Jan J. Dijkstra and the author in \cite{lip5}. Essentially,  we use the Axiom of Choice to show  $\textnormal{\textbf{\d{$\nabla$}}}X$ (Example 1) contains a connected set with an explosion point. 

 


\begin{thebibliography}{HD}

\bibitem{cookproblem}H. Cook, W.T. Ingram, A. Lelek, A list of problems known as Houston Problem Book, in: H. Cook, W.T. Ingram, K.T. Kuperberg, A. Lelek, P. Minc (Eds.), Continua with The Houston Problem Book, in: Lecture Notes in Pure and Applied Mathematics, vol. 170, Marcel Dekker, New York, Basel, Hong Kong, 1995, pp. 365--398.
\bibitem{deb}W. D\k{e}bski, Note on a problem by H. Cook. Houston J. Math. 17 (1991) 439--441.
\bibitem{lip5}J.J. Dijkstra and D.S. Lipham, A note on cohesive almost zero-dimensional space, preprint.\bibitem{eng1}R. Engelking, Dimension Theory, Volume 19 of Mathematical Studies, North-Holland Publishing Company, 1978.
\bibitem{note}S.D. Iliadis, A note on compactifications of rim-scattered spaces, Topology Proc. 5th Int. Meet., Lecce/Italy 1990, Suppl. Rend. Circ. Mat. Palermo, II. Ser. 29, (1992) 425--433.
\bibitem{ili}S.D. Iliadis and E.D. Tymchatyn, Compactifications with minimum rim-types of rational spaces, Houston J. Math. 17(3) (1991) 311--323.
\bibitem{jon}F.B. Jones, Wilder on connectedness. Algebraic and geometric topology (Proc. Sympos., Univ. California, Santa Barbara, Calif., 1977), pp. 1--6, 
Lecture Notes in Math., 664, Springer, Berlin, 1978.
\bibitem{kk}B. Knaster and C. Kuratowski,  Sur les ensembles connexes, Fundamenta Math. 2 (1921) 206--255.
\bibitem{sa}E. Lamb, \textit{A Few of My Favorite Spaces: Cantor's Leaky Tent},  Scientific American,\\ \url{https://blogs.scientificamerican.com/roots-of-unity/a-few-of-my-favorite-spaces-cantor-s-leaky-tent/}, June 20 2015.
\bibitem{lel}A. Lelek, On the topology of curves II,  Fund. Math. 70 (1971), 131--138.
\bibitem{lip}D.S. Lipham, Widely-connected sets in the bucket-handle continuum,  Fund. Math. 240 (2018) 161--174.
\bibitem{mil} E.W. Miller, Concerning biconnected sets, Fundamenta Math. 29 (1937) 123--133.
\bibitem{minc}P. Minc, Countable subsets of Suslinian continua, Fund. Math. 124 (1984), 123--129.
\bibitem{swi} P.M. Swingle, Two types of connected sets, Bull. Amer. Math. Soc. 37 (1931) 254--258.
\bibitem{tym}E.D. Tymchatyn, Compactifications of rational spaces, Houston J. Math. 3 (1977), no. 1, 131--139.
\bibitem{wil}R.L. Wilder, Evolution of the topological concept of ``connected'', Amer. Math. Monthly 85 (1978), 720--726. 

\end{thebibliography}
\end{document}